\documentclass{amsart}
\usepackage{amssymb}
\usepackage{amsfonts}
\usepackage{hyperref}

\newtheorem{theorem}{Theorem}
\theoremstyle{plain}

\newtheorem{conjecture}{Conjecture}
\newtheorem{corollary}{Corollary}

\newtheorem{lemma}{Lemma}

\newtheorem{problem}{Problem}

\newtheorem{remark}{Remark}

\numberwithin{equation}{section}
\input{tcilatex}

\begin{document}
\title[Several completely monotone functions]{Several completely monotone
functions related to DeTemple's sequence}
\author{Zhen-Hang Yang}
\address{Power Supply Service Center, ZPEPC Electric Power Research
Institute, Hangzhou, Zhejiang, China, 310009}
\email{yzhkm@163.com}
\date{July 15, 2014}
\thanks{This paper is in final form and no version of it will be submitted
for publication elsewhere.}

\begin{abstract}
In this paper, we present the necessary and sufficient conditions such that
several functions involving $R\left( x\right) =\psi \left( x+1/2\right) -\ln
x$ with a parameter are completely monotone on $\left( 0,\infty \right) $,
where $\psi $ is the digamma function. This generalizes some known results
and verifies a conjecture posed by Chen.
\end{abstract}

\maketitle

\section{Introduction}

A function $f$ is said to be completely monotonic on an interval $I$ , if $f$
has derivatives of all orders on $I$ and satisfies

\begin{equation}
(-1)^{n}f^{(n)}(x)\geq 0\text{ for all }x\in I\text{ and }n=0,1,2,....
\label{cm}
\end{equation}

If the inequality (\ref{cm}) is strict, then $f$ is said to be strictly
completely monotonic on $I$. It is known (Bernstein's Theorem) that $f$ is
completely monotonic on $(0,\infty )$ if and only if

\begin{equation*}
f(x)=\int_{0}^{\infty }e^{-xt}d\mu \left( t\right) ,
\end{equation*}%
where $%
%TCIMACRO{\U{3bc} }%
%BeginExpansion
\mu
%EndExpansion
$ is a nonnegative measure on $[0,\infty )$ such that the integral converges
for all $x>0$, see \cite[p. 161]{Widder-PUPP-1941}.

For $x>0$ the classical Euler's gamma function $\Gamma $ and psi (digamma)
function $\psi $ are defined by%
\begin{equation}
\Gamma \left( x\right) =\int_{0}^{\infty }t^{x-1}e^{-t}dt,\text{ \ \ \ \ }%
\psi \left( x\right) =\frac{\Gamma ^{\prime }\left( x\right) }{\Gamma \left(
x\right) },  \label{Gamma}
\end{equation}%
respectively. The derivatives $\psi ^{\prime }$, $\psi ^{\prime \prime }$, $%
\psi ^{\prime \prime \prime }$, ... are known as polygamma functions.

As the important role played in many branches, such as mathematical physics,
probability, statistics, engineering, the gamma and polygamma functions have
attracted the attention of many scholars. In particular many authors
published numerous interesting inequalities for the Euler-Mascheroni
constant defined by%
\begin{equation*}
\gamma =\lim_{n\rightarrow \infty }D_{n}:=\lim_{n\rightarrow \infty }\left(
\sum_{k=1}^{n}\frac{1}{k}-\ln n\right) =0.577215664...,
\end{equation*}%
as well as digamma (psi) function $\psi $ due to $\psi \left( n+1\right)
=H_{n}-\gamma $.

It is known that the sequence $D_{n}$ converges very slowly to its limit,
due to%
\begin{equation*}
\frac{1}{2\left( n+1\right) }<D_{n}-\gamma <\frac{1}{2n}
\end{equation*}%
(see \cite{Rippon-AMM-93-1986}, \cite{Young-MG-75-1991}). Some quicker
approximations to the Euler-Mascheroni constant were established in \cite%
{Tims-MG-55-1971}, \cite{Toth-AMM-98-1991}, \cite{Toth-AMM-99-1992}, \cite%
{DeTemple-AMM-100-1993}, \cite{Negoi-GM-15-1997}, \cite%
{Vernescu-GMSA-17(96)-1999}, \cite{Qi-JMAA-310(1)-2005}, \cite%
{Sintamarian-NA-46-2007}, \cite{Chen-GJAMMS-1(1)-2008}, \cite%
{Sintamarian-JIPAM-9(2)-2008}, \cite{Villarino-JIPAM-9(3)-2008}, \cite%
{Chen-JMI-3(1)-2009}, \cite{Chen-RRC-12(3)-2009}, \cite{Chen-AMC-217-2010}, 
\cite{Chen-AML-23-2010}, \cite{Mortici-CMA-59-2010}, \cite%
{Mortici-AMC-215-2010}, \cite{Mortici-CMA-29-2010}, \cite%
{Guo-AMC-218(3)-2011}, \cite{Chen-BMAA-3-2011}, \cite{Mortici-NA-56-2011}, 
\cite{Chen-CMA-64-2013}, \cite{Gavrea-AMC-224-2013}, \cite{Lu-JMAA-419-2014}.

In 1993, DeTemple \cite{DeTemple-AMM-100-1993} defined a modified sequence 
\begin{equation}
R_{n}=\sum_{k=1}^{n}\frac{1}{k}-\ln \left( n+\frac{1}{2}\right)  \label{R_n}
\end{equation}%
to approximate to $\gamma $, which converges faster since%
\begin{equation}
\frac{1}{24\left( n+1\right) ^{2}}<R_{n}-\gamma <\frac{1}{24n^{2}},
\label{De1}
\end{equation}%
and%
\begin{equation}
\frac{7}{960}\frac{1}{\left( n+1\right) ^{4}}<R_{n}-\gamma -\frac{1}{24n^{2}}%
<\frac{7}{960}\frac{1}{n^{4}}.  \label{De2}
\end{equation}%
Villarino \cite[Theorem 1.7]{Villarino-JIPAM-9(3)-2008} proved that for $%
n\in \mathbb{N}$, the following inequality%
\begin{equation}
\tfrac{1}{24\left( n+1/2\right) ^{2}+21/5}<R_{n}-\gamma <\tfrac{1}{24\left(
n+1/2\right) ^{2}+1/\left( 1-\ln 3+\ln 2-\gamma \right) -54}  \label{V}
\end{equation}%
is valid with the best constants $21/5$ and $1/\left( 1-\ln 3+\ln 2+\func{Psi%
}\left( 1\right) \right) -54\approx 3.7393$. Chen \cite{Chen-AML-23-2010}
proved that%
\begin{equation}
\tfrac{1}{24}\left( n+\lambda \right) ^{-2}<R_{n}-\gamma <\frac{1}{24}\left(
n+\frac{1}{2}\right) ^{-2},  \label{Ch}
\end{equation}%
where%
\begin{equation*}
\lambda =\frac{1}{2\sqrt{6\left( 1+\func{Psi}\left( 1\right) -\ln 3+\ln
2\right) }}-1\approx 0.551\,07
\end{equation*}
and $1/2$ are the best constants.

Mortici \cite[Theorem 2.1]{Mortici-JSA-2(13)-2010} gave new bounds as
follows:%
\begin{equation}
\frac{1}{24}\left( n+\frac{1}{2}+\frac{7}{80n}\right) ^{-2}<R_{n}-\gamma <%
\frac{1}{24}\left( n+\frac{1}{2}\right) ^{-2}.  \label{M}
\end{equation}

In \cite{Batir-AM-91-2008}, \cite{Qi-arxiv-0902.2524} (also see \cite%
{Guo-A-34-2014}), it was established that%
\begin{equation*}
\gamma +\ln \left( n+\tfrac{1}{2}\right) <\sum_{k=1}^{n}\frac{1}{k}\leq
\gamma +\ln \left( n+e^{1-\gamma }-1\right) ,
\end{equation*}%
which is equivalent to%
\begin{equation*}
0<R_{n}-\gamma \leq \ln \frac{n+e^{1-\gamma }-1}{n+1/2}.
\end{equation*}

On the other hand, Karatsuba \cite{Karatsuba-NA-24-2000} proved that for all
integers $n\geq 1$, $H(n)<H(n+1)$, where $H(n)$ is defined by%
\begin{equation}
H(n)=\left( R_{n}-\gamma \right) n^{2}.  \label{H(n)}
\end{equation}%
In view of $\psi \left( n+1\right) =H_{n}-\gamma $, we see that%
\begin{equation*}
H(n)=\left( R_{n}-\gamma \right) n^{2}=\left( \psi \left( n+1\right) -\ln
\left( n+\frac{1}{2}\right) \right) n^{2}.
\end{equation*}%
As mentioned in \cite{Anderson-PA-2001}, some computer experiments also seem
to indicate that $(1+1/n)^{2}H(n)$ is a decreasing convex function. Chen 
\cite[Theorem 2]{Chen-JMI-3(1)-2009} proved that for all integers $n\geq
1,H(n)$ and $(\left( n+1/2\right) /n)^{2}H(n)$ are both strictly increasing
concave sequences, while $[\left( (n+1)/n\right) ^{2}H(n)$ is strictly
decreasing convex sequence. Also, he conjectured in \cite[Cojecture 1]%
{Chen-JMI-3(1)-2009} that

(i) the functions $H(x)$ and $(\left( x+1/2\right) /x)^{2}H(x)$ are both
so-called Bernstein function on $(0,\infty )$. That is,%
\begin{eqnarray*}
H(x) &>&0\text{, \ }\left( -1\right) ^{n}\left[ H(x)\right] ^{\left(
n+1\right) }>0, \\
(\left( x+1/2\right) /x)^{2}H(x) &>&0\text{, \ }\left( -1\right) ^{n}\left[
(\left( x+1/2\right) /x)^{2}H(x)\right] ^{\left( n+1\right) }>0
\end{eqnarray*}%
hold for $x>0$ and $n\in \mathbb{N}$.

(ii) The function $\left( (x+1)/x\right) ^{2}H(x)$ is strictly completely
monotonic on $(0,\infty )$.

We remark that the functions $H(x)$ is not Bernstein function on $(0,\infty
) $ due to%
\begin{eqnarray*}
-H^{\prime \prime }(0^{+}) &=&2\ln 2-2\gamma \approx -0.231\,86<0, \\
-H^{\prime \prime }(1/2) &=&2\gamma +4\ln 2+\frac{7}{2}\zeta \left( 3\right)
-\pi ^{2}+\frac{7}{4}\approx 0.014615>0.
\end{eqnarray*}

The aim of this paper is to prove the complete monotonicity of certain
functions involving%
\begin{equation}
R\left( x\right) =\psi \left( x+1/2\right) -\ln x.  \label{R(x)}
\end{equation}%
In Section 2, some useful lemmas are given. Our main results are presented
in Section 3, in which Theorem \ref{MT-ha} indicates that the second
conjecture posed by Chen is valid, while Theorems \ref{MT-Fa}--\ref{MT-Ga}
show the necessary and sufficient conditions such that three functions
involving $R\left( x\right) $ with a parameter are completely monotone on $%
\left( 0,\infty \right) $. In the last section, some remarks, conjecture and
open problem are presented.

\section{Lemmas}

In order to prove our results, we need some lemmas.

\begin{lemma}
\label{Lemma Rip}For $t>0$, let the function $Q$ defined by%
\begin{equation}
Q\left( t\right) =\frac{1}{t}-\frac{1}{2\sinh \frac{t}{2}}.  \label{Q}
\end{equation}%
Then%
\begin{eqnarray}
R\left( x\right) &=&\int_{0}^{\infty }e^{-xt}Q\left( t\right) dt,  \label{R0}
\\
xR\left( x\right) &=&\int_{0}^{\infty }e^{-xt}Q^{\prime }\left( t\right) dt,
\label{R1} \\
x^{2}R\left( x\right) &=&\frac{1}{24}+\int_{0}^{\infty }e^{-xt}Q^{\prime
\prime }\left( t\right) dt,  \label{R2} \\
x^{3}R\left( x\right) &=&\frac{1}{24}x+\int_{0}^{\infty }e^{-xt}Q^{\prime
\prime \prime }\left( t\right) dt,  \label{R3} \\
x^{4}R\left( x\right) &=&\frac{1}{24}x^{2}-\frac{7}{960}+\int_{0}^{\infty
}e^{-xt}Q^{(4)}\left( t\right) dt.  \label{R4}
\end{eqnarray}
\end{lemma}

\begin{proof}
Using the integral representations \cite[p. 259]{Abramowitz-AMS-W-1965}%
\begin{equation*}
\psi (x)=\int_{0}^{\infty }\left( \frac{e^{-t}}{t}-\frac{e^{-xt}}{1-e^{-t}}%
\right) dt\text{ \ \ \ and \ \ \ }\ln x=\int_{0}^{\infty }\frac{%
e^{-t}-e^{-xt}}{t}dt
\end{equation*}%
gives%
\begin{eqnarray*}
R\left( x\right) &=&\psi (x+\tfrac{1}{2})-\ln x=\int_{0}^{\infty }\left( 
\frac{e^{-xt}}{t}-\frac{e^{-(x+1/2)t}}{1-e^{-t}}\right) dt \\
&=&\int_{0}^{\infty }e^{-xt}\left( \frac{1}{t}-\frac{1}{2\sinh \frac{t}{2}}%
\right) dt=\int_{0}^{\infty }e^{-xt}Q\left( t\right) dt.
\end{eqnarray*}%
An integration by part yields%
\begin{eqnarray*}
xR\left( x\right) &=&x\int_{0}^{\infty }e^{-xt}Q\left( t\right)
dt=-\int_{0}^{\infty }Q\left( t\right) de^{-xt} \\
&=&-e^{-xt}Q\left( t\right) |_{0}^{\infty }+\int_{0}^{\infty
}e^{-xt}Q^{\prime }\left( t\right) dt=\int_{0}^{\infty }e^{-xt}Q^{\prime
}\left( t\right) dt,
\end{eqnarray*}%
where the last equality holds due to $\lim_{t\rightarrow \infty }\left(
e^{-xt}Q\left( t\right) \right) =\lim_{t\rightarrow 0}\left( e^{-xt}Q\left(
t\right) \right) =0$.

Integrations by parts again and noticing that%
\begin{equation*}
\lim_{t\rightarrow \infty }\left( e^{-xt}Q^{\prime }\left( t\right) \right)
=0\text{ \ and \ }\lim_{t\rightarrow 0}\left( e^{-xt}Q^{\prime }\left(
t\right) \right) =-\frac{1}{24}
\end{equation*}%
lead to%
\begin{eqnarray*}
x^{2}R\left( x\right) &=&x\int_{0}^{\infty }e^{-xt}Q^{\prime }\left(
t\right) dt=-e^{-xt}Q^{\prime }\left( t\right) |_{0}^{\infty
}+\int_{0}^{\infty }e^{-xt}Q^{\prime \prime }\left( t\right) dt \\
&=&\frac{1}{24}+\int_{0}^{\infty }e^{-xt}Q^{\prime \prime }\left( t\right)
dt.
\end{eqnarray*}

Similarly, we have%
\begin{eqnarray*}
x^{3}R\left( x\right) &=&\frac{1}{24}x+x\int_{0}^{\infty }e^{-xt}Q^{\prime
\prime }\left( t\right) dt=\frac{1}{24}x-e^{-xt}Q^{\prime \prime }\left(
t\right) |_{0}^{\infty }+\int_{0}^{\infty }e^{-xt}Q^{\prime \prime \prime
}\left( t\right) dt \\
&=&\frac{1}{24}x+\int_{0}^{\infty }e^{-xt}Q^{\prime \prime \prime }\left(
t\right) dt
\end{eqnarray*}%
due to $\lim_{t\rightarrow \infty }\left( e^{-xt}Q^{\prime \prime }\left(
t\right) \right) =\lim_{t\rightarrow 0}\left( e^{-xt}Q^{\prime \prime
}\left( t\right) \right) =0$ and%
\begin{eqnarray*}
x^{4}R\left( x\right) &=&\frac{1}{24}x^{2}+x\int_{0}^{\infty
}e^{-xt}Q^{\prime \prime \prime }\left( t\right) dt=\frac{1}{24}%
x^{2}-e^{-xt}Q^{\prime \prime \prime }\left( t\right) |_{0}^{\infty
}+\int_{0}^{\infty }e^{-xt}Q^{(4)}\left( t\right) dt \\
&=&\frac{1}{24}x^{2}-\frac{7}{960}+\int_{0}^{\infty }e^{-xt}Q^{(4)}\left(
t\right) dt
\end{eqnarray*}%
due to%
\begin{equation*}
\lim_{t\rightarrow \infty }\left( e^{-xt}Q^{\prime \prime \prime }\left(
t\right) \right) =0\text{ \ and \ }\lim_{t\rightarrow 0}\left(
e^{-xt}Q^{\prime \prime \prime }\left( t\right) \right) =-\frac{7}{960}.
\end{equation*}%
This completes the proof.
\end{proof}

In the same method as Lemma 5 in \cite{Yang-AAA-2014-702718}, we can prove
the following statement.

\begin{lemma}[{\protect\cite[Lemma 5]{Yang-AAA-2014-702718}}]
\label{Lemma zp}Let $P\left( t\right) $ be a power series which is
convergent on $\left( 0,\infty \right) $ defined by%
\begin{equation*}
P\left( t\right) =\sum_{i=m+1}^{\infty }a_{i}t^{i}-\sum_{i=0}^{m}a_{i}t^{i},
\end{equation*}%
where $a_{i}\geq 0$ for $i\geq m+1$ with $\max_{i\geq m+1}\left(
a_{i}\right) >0$ and $a_{m}>0$, $a_{i}\geq 0$ for $0\leq i\leq m-1$. Then
there is a unique number $t_{0}\in \left( 0,\infty \right) $ to satisfy $%
P\left( t\right) =0$ such that $P\left( t\right) <0$ for $t\in \left(
0,t_{0}\right) $ and $P\left( t\right) >0$ for $t\in \left( t_{0},\infty
\right) $.
\end{lemma}

Using Lemma \ref{Lemma zp} we can prove the following

\begin{lemma}
\label{Lemma dQ/Q}Let the function $Q$ defined by (\ref{Q}). Then $Q^{\prime
}\left( t\right) \geq c_{0}Q\left( t\right) $ for $t>0$, where%
\begin{equation}
c_{0}=\min_{t>0}\left( \frac{Q^{\prime }\left( t\right) }{Q\left( t\right) }%
\right) \approx -0.061875.  \label{c0}
\end{equation}
\end{lemma}

\begin{proof}
Differentiation yields%
\begin{equation*}
\frac{Q^{\prime }\left( t\right) }{Q\left( t\right) }=\frac{\frac{1}{4}\frac{%
\cosh \frac{t}{2}}{\sinh ^{2}\frac{t}{2}}-\frac{1}{t^{2}}}{\frac{1}{t}-\frac{%
1}{2\sinh \frac{t}{2}}}=\frac{t^{2}\cosh \frac{t}{2}-4\sinh ^{2}\frac{t}{2}}{%
4t\sinh ^{2}\frac{t}{2}-2t^{2}\sinh \frac{t}{2}},
\end{equation*}%
\begin{eqnarray*}
\left( \frac{Q^{\prime }\left( t\right) }{Q\left( t\right) }\right) ^{\prime
} &=&\tfrac{1}{4t^{2}\left( t-2\sinh \frac{t}{2}\right) ^{2}\sinh ^{2}\frac{t%
}{2}}(-16t\sinh ^{3}\frac{t}{2}+t^{4}\cosh ^{2}\frac{t}{2}+2t^{3}\sinh ^{3}%
\frac{t}{2} \\
&&-t^{4}\sinh ^{2}\frac{t}{2}+16\sinh ^{4}\frac{t}{2}+8t^{2}\cosh \frac{t}{2}%
\sinh ^{2}\frac{t}{2}-4t^{3}\cosh ^{2}\frac{t}{2}\sinh \frac{t}{2}) \\
&:&=\frac{p\left( \frac{t}{2}\right) }{4t^{2}\left( t-2\sinh \frac{t}{2}%
\right) ^{2}\sinh ^{2}\frac{t}{2}},
\end{eqnarray*}%
where%
\begin{eqnarray*}
p\left( t\right) &=&16(t^{4}\cosh ^{2}t+t^{3}\sinh ^{3}t-t^{4}\sinh
^{2}t+\sinh ^{4}t-2t\sinh ^{3}t \\
&&+2t^{2}\cosh t\sinh ^{2}t-2t^{3}\cosh ^{2}t\sinh t.
\end{eqnarray*}%
If we prove that there is a unique number $t_{0}>0$ such that $Q^{\prime }/Q$
is decreasing on $\left( 0,t_{0}\right) $ and increasing on $\left(
t_{0},\infty \right) $, then we have $Q^{\prime }\left( t\right) /Q\left(
t\right) \geq Q^{\prime }\left( t_{0}\right) /Q\left( t_{0}\right) $ for all 
$t>0$. For this end, we use the "product into sum" formulas and Taylor
expansion to get%
\begin{eqnarray*}
\frac{1}{2}p\left( t\right) &=&\cosh 4t+4t^{2}\cosh 3t-2t^{3}\sinh
3t-4t\sinh 3t-4\cosh 2t \\
&&-4t^{2}\cosh t-10t^{3}\sinh t+12t\sinh t+8t^{4}+3 \\
&=&\sum_{n=0}^{\infty }\frac{4^{2n}t^{2n}}{\left( 2n\right) !}%
+4\sum_{n=1}^{\infty }\frac{3^{2n-2}t^{2n}}{\left( 2n-2\right) !}%
-2\sum_{n=2}^{\infty }\frac{3^{2n-3}t^{2n}}{\left( 2n-3\right) !}%
-4\sum_{n=1}^{\infty }\frac{3^{2n-1}t^{2n}}{\left( 2n-1\right) !} \\
&&-4\sum_{n=0}^{\infty }\frac{2^{2n}t^{2n}}{\left( 2n\right) !}%
-4\sum_{n=1}^{\infty }\frac{t^{2n}}{\left( 2n-2\right) !}-10\sum_{n=2}^{%
\infty }\frac{t^{2n}}{\left( 2n-3\right) !} \\
&&+12\sum_{n=1}^{\infty }\frac{t^{2n}}{\left( 2n-1\right) !}+8t^{4}+3 \\
&:&=\sum_{n=3}^{\infty }\frac{u_{n}}{\left( 2n\right) !}t^{2n},
\end{eqnarray*}%
where%
\begin{equation*}
u_{n}=\left( 4^{2n}-8n\left( 2n^{2}-9n+13\right) 3^{2n-3}-2^{2n+2}-8n\left(
10n^{2}-13n+1\right) \right) .
\end{equation*}%
A direct verification gives $u_{3}=0$, $u_{n}<0$ for $4\leq n\leq 10$, and
we now show that $u_{n}\geq 0$ for $n\geq 11$. It is easy to check that $%
u_{n}$ satisfies the recursive relation%
\begin{eqnarray*}
u_{n+1}-16u_{n} &=&8\left( 14n^{3}-117n^{2}+199n-54\right) 3^{2n-3}+48\times
2^{2n} \\
&&+1200n^{3}-1800n^{2}+88n+16 \\
&>&0\text{ for }n\geq 11\text{, }
\end{eqnarray*}%
which together with $u_{11}=1636\,643\,754\,240>0$ indicates that $u_{n}\geq
0$ for $n\geq 11$.

By Lemma \ref{Lemma zp} we see that there is a unique $t_{0}^{\prime }>0$
such that $p\left( t\right) <0$ for $t\in \left( 0,t_{0}^{\prime }\right) $
and $p\left( t\right) <0$ for $t\in \left( t_{0}^{\prime },\infty \right) $,
which implies that $Q^{\prime }/Q$ is decreasing on $\left( 0,2t_{0}^{\prime
}\right) $ and increasing on $\left( 2t_{0}^{\prime },\infty \right) $.
Numeric calculation shows that $t_{0}=2t_{0}^{\prime }=15.40151...$, and so $%
Q^{\prime }\left( t_{0}\right) /Q\left( t_{0}\right) \approx -0.061875$.

This completes the proof.
\end{proof}

To simplify the proofs of some crucial inequalities, we need to the
following lemma.

\begin{lemma}
\label{Lemma sinht/t}For $t>0$, the following inequalities hold:%
\begin{eqnarray}
\frac{\sinh t}{t} &>&3\frac{2\cosh t+3}{\cosh t+14},  \label{sinht/t>1} \\
\frac{\sinh t}{t} &>&15\frac{2\cosh ^{2}t+10\cosh t+9}{2\cosh ^{2}t+101\cosh
t+212},  \label{sinht/t>2} \\
\frac{\sinh t}{t} &<&15\frac{18\cosh ^{2}t+160\cosh t+179}{1159\cosh
^{2}t+4192\cosh t+4}\cosh t.  \label{sinht/t<1}
\end{eqnarray}
\end{lemma}

\begin{proof}
The inequality (\ref{sinht/t>1}) has been proven in \cite[Theorem 18]%
{Yang-AAA-2014-364076}.

To prove (\ref{sinht/t>2}), it suffices to show that for $t>0$%
\begin{equation*}
p_{1}\left( t\right) :=\frac{2\cosh ^{2}t+101\cosh t+212}{2\cosh
^{2}t+10\cosh t+9}\sinh t-15t>0.
\end{equation*}%
Differentiating and factoring give%
\begin{eqnarray*}
p_{1}^{\prime }\left( t\right) &=&\tfrac{2\cosh ^{2}t+101\cosh t+212}{2\cosh
^{2}t+10\cosh t+9}\cosh t-7\tfrac{26\cosh ^{2}t+116\cosh t+173}{\left(
2\cosh ^{2}t+10\cosh t+9\right) ^{2}}\sinh ^{2}t-15 \\
&=&\frac{4\left( \cosh t-1\right) ^{5}}{\left( 2\cosh ^{2}t+10\cosh
t+9\right) ^{2}}>0,
\end{eqnarray*}%
which yields $p_{1}\left( t\right) >p_{1}\left( 0\right) =0$.

Similarly, inequality (\ref{sinht/t<1}) is equivalent to%
\begin{equation*}
p_{2}\left( t\right) :=\frac{1159\cosh ^{2}t+4192\cosh t+4}{\left( 18\cosh
^{2}t+160\cosh t+179\right) \cosh t}\sinh t-15t<0.
\end{equation*}%
Differentiating and factoring lead us to%
\begin{eqnarray*}
p_{2}^{\prime }\left( t\right) &=&\tfrac{1159\cosh ^{2}t+4192\cosh t+4}{%
\left( 18\cosh ^{2}t+160\cosh t+179\right) \cosh t}\cosh t+\left( \sinh
^{2}t\right) \frac{d}{dx}\left( \tfrac{1159x^{2}+4192x+4}{\left(
18x^{2}+160x+179\right) x}\right) -15 \\
&=&-\frac{4\left( 1215x+179\right) \left( x-1\right) ^{5}}{x^{2}\left(
18x^{2}+160x+179\right) ^{2}}<0,
\end{eqnarray*}%
where $x=\cosh t>1$. This implies that $p_{2}\left( t\right) <p_{2}\left(
0\right) =0$.
\end{proof}

\begin{lemma}
\label{Lemma ddQ/Q>}The function $Q$ defined by (\ref{Q}) satisfies%
\begin{equation*}
q_{1}\left( t\right) :=Q^{\prime \prime }\left( t\right) +\frac{7}{40}%
Q\left( t\right) >0
\end{equation*}%
for all $t>0$.
\end{lemma}

\begin{proof}
Differentiation yields%
\begin{equation}
Q^{\prime \prime }\left( t\right) =\frac{1}{8\sinh \frac{t}{2}}-\frac{1}{4}%
\frac{\cosh ^{2}\frac{t}{2}}{\sinh ^{3}\frac{t}{2}}+\frac{2}{t^{3}},
\label{ddQ}
\end{equation}%
and then%
\begin{equation*}
q_{1}\left( t\right) =-\frac{1}{80}\frac{20t^{3}\cosh ^{2}\frac{t}{2}%
-14t^{2}\sinh ^{3}\frac{t}{2}-3t^{3}\sinh ^{2}\frac{t}{2}-160\sinh ^{3}\frac{%
t}{2}}{t^{3}\sinh ^{3}\frac{t}{2}}.
\end{equation*}%
To prove the desired inequality, we write $q_{1}\left( 2t\right) $ in the
form of 
\begin{equation*}
\left( 80\sinh ^{3}t\right) q_{1}\left( 2t\right) =20\left( \frac{\sinh t}{t}%
\right) ^{3}+7\left( \sinh ^{2}t\right) \frac{\sinh t}{t}+3\sinh
^{2}t-20\cosh ^{2}t.
\end{equation*}%
Employing inequality (\ref{sinht/t>1}), we get%
\begin{eqnarray*}
\left( 80\sinh ^{3}t\right) q_{1}\left( 2t\right) &>&20\left( 3\frac{2\cosh
t+3}{\cosh t+14}\right) ^{3}+7\left( \cosh ^{2}t-1\right) \left( 3\frac{%
2\cosh t+3}{\cosh t+14}\right) \\
&&+3\left( \cosh ^{2}t-1\right) -20\cosh ^{2}t \\
&=&25\left( \cosh ^{2}t+24\cosh t+240\right) \frac{\left( \cosh t-1\right)
^{3}}{\left( \cosh t+14\right) ^{3}}>0,
\end{eqnarray*}%
which proves the desired inequality.
\end{proof}

\begin{lemma}
\label{Lemma ddddQ/Q<}The function $Q$ defined by (\ref{Q}) satisfies%
\begin{equation*}
q_{2}\left( t\right) :=Q^{(4)}\left( t\right) -\frac{31}{336}Q\left(
t\right) <0
\end{equation*}%
for all $t>0$.
\end{lemma}

\begin{proof}
Differentiation yields%
\begin{equation}
Q^{(4)}\left( t\right) =\frac{7}{8}\frac{\cosh ^{2}\frac{1}{2}t}{\sinh ^{3}%
\frac{1}{2}t}-\frac{3}{4}\frac{\cosh ^{4}\frac{1}{2}t}{\sinh ^{5}\frac{1}{2}t%
}-\frac{5}{32\sinh \frac{1}{2}t}+\frac{24}{t^{5}},  \label{ddddQ}
\end{equation}%
and then%
\begin{equation*}
q_{2}\left( t\right) =-\tfrac{1}{336}\tfrac{252t^{5}\cosh ^{4}\frac{t}{2}%
+31t^{4}\sinh ^{5}\frac{t}{2}+37t^{5}\sinh ^{4}\frac{t}{2}-8064\sinh ^{5}%
\frac{t}{2}-294t^{5}\cosh ^{2}\frac{t}{2}\sinh ^{2}\frac{t}{2}}{t^{5}\sinh
^{5}\frac{t}{2}}.
\end{equation*}%
A simple identical transformation gives%
\begin{eqnarray*}
-\left( 672\sinh ^{5}t\right) q_{2}\left( 2t\right) &=&-504\left( \frac{%
\sinh t}{t}\right) ^{5}+31\left( \sinh ^{4}t\right) \frac{\sinh t}{t} \\
&&+504\cosh ^{4}t-588\cosh ^{2}t\sinh ^{2}t+74\sinh ^{4}t.
\end{eqnarray*}%
In order to prove the desired inequality, we define%
\begin{equation*}
U\left( y\right) =-504y^{5}+31(\sinh ^{4}t)y+504\cosh ^{4}t-588\cosh
^{2}t\sinh ^{2}t+74\sinh ^{4}t,
\end{equation*}%
and it suffices to prove that $U\left( \left( \sinh t\right) /t\right) >0$
for $t>0$.

An easy computation yields $U^{\prime }\left( y\right) =31\sinh
^{4}t-2520y^{4}$, which reveals that $U$ is increasing for $y\in \left( 1,%
\sqrt[4]{31/2520}\sinh t\right) $ and decreasing for $y\in \lbrack \sqrt[4]{%
31/2520}\sinh t,\infty )$. We now distinguish two cases to prove $U\left(
\left( \sinh t\right) /t\right) >0$ for $t>0$.

In the case when $t\in \left( \sqrt[4]{2520/31},\infty \right) $, by
inequality (\ref{sinht/t>1}), we have%
\begin{equation*}
1<3\frac{2\cosh t+3}{\cosh t+14}<\frac{\sinh t}{t}<\sqrt[4]{\frac{31}{2520}}%
\sinh t,
\end{equation*}%
that is,%
\begin{equation*}
3\frac{2\cosh t+3}{\cosh t+14},\frac{\sinh t}{t}\in \left( 1,\sqrt[4]{31/2520%
}\sinh t\right) ,
\end{equation*}%
and so%
\begin{eqnarray*}
U\left( \frac{\sinh t}{t}\right) &>&U\left( 3\frac{2\cosh t+3}{\cosh t+14}%
\right) =\left( 504\cosh ^{4}t-588\cosh ^{2}t\sinh ^{2}t+74\sinh ^{4}t\right)
\\
&&+31\left( \sinh ^{4}t\right) \times 3\frac{2\cosh t+3}{\cosh t+14}%
-504\left( 3\frac{2\cosh t+3}{\cosh t+14}\right) ^{5}.
\end{eqnarray*}%
Putting $\cosh t=x$, then $\sinh ^{2}t=x^{2}-1$, and factoring yield%
\begin{equation*}
U\left( \frac{\sinh t}{t}\right) >\frac{\left( x-1\right) ^{3}}{\left(
x+14\right) ^{5}}U_{1}\left( x\right)
\end{equation*}%
where%
\begin{eqnarray*}
U_{1}\left( x\right) &=&176x^{6}+10\,523x^{5}+245\,869x^{4}+2810\,864x^{3} \\
&&+12\,467\,224x^{2}+12\,511\,688x-20\,756\,344.
\end{eqnarray*}%
It is easy to verify that $U_{1}\left( x\right) >U_{1}\left( 1\right)
=7290\,000>0$, which implies that $U\left( \left( \sinh t\right) /t\right)
>0 $ for $t\in \left( \sqrt[4]{2520/31},\infty \right) $.

In the case of when $t\in (0,\sqrt[4]{2520/31}]$, from inequality (\ref%
{sinht/t<1}) we see that%
\begin{equation*}
\infty >15\frac{18\cosh ^{2}t+160\cosh t+179}{1159\cosh ^{2}t+4192\cosh t+4}%
\cosh t>\frac{\sinh t}{t}>\sqrt[4]{\frac{31}{2520}}\sinh t.
\end{equation*}%
Using the decreasing property of $U$ for $u\in \left( \sqrt[4]{31/2520}\sinh
t,\infty \right) $, we have%
\begin{equation*}
U\left( \frac{\sinh t}{t}\right) >U\left( 15\frac{18\cosh ^{2}t+160\cosh
t+179}{1159\cosh ^{2}t+4192\cosh t+4}\cosh t\right) .
\end{equation*}%
Letting $\cosh t=x$ and factoring indicate that%
\begin{equation*}
U\left( 15\frac{18x^{2}+160x+179}{1159x^{2}+4192x+4}x\right) =\frac{\left(
x-1\right) ^{4}}{\left( 1159x^{2}+4192x+4\right) ^{5}}U_{2}\left( x\right) ,
\end{equation*}%
where%
\begin{eqnarray*}
U_{2}\left( x\right)
&=&14\,379\,675\,269\,523\,570x^{11}+357\,214\,567\,270\,415\,330x^{10} \\
&&+3604\,910\,878\,299\,956\,955x^{9}+19\,027\,526\,850\,473\,930\,600x^{8}
\\
&&+55\,570\,610\,110\,726\,848\,080x^{7}+85\,295\,682\,448\,077\,545%
\,696x^{6} \\
&&+54\,079\,668\,524\,631\,977\,864x^{5}+560\,130\,320\,580\,220\,160x^{4} \\
&&+1016\,873\,963\,329\,280x^{3}+923\,378\,178\,560x^{2}+418\,677\,504x+75%
\,776.
\end{eqnarray*}%
Evidently, $U_{2}\left( x\right) >0$, and so $U\left( \left( \sinh t\right)
/t\right) >0$ for $t\in (0,\sqrt[4]{2520/31}]$.

This proof is completed.
\end{proof}

\begin{lemma}
\label{Lemma ddddQ-ddQ-Q}The function $Q$ defined by (\ref{Q}) satisfies%
\begin{equation*}
q_{3}\left( t\right) :=Q^{(4)}\left( t\right) +\frac{11\,165}{8284}Q^{\prime
\prime }\left( t\right) +\frac{199\,849}{1391\,712}Q\left( t\right) >0
\end{equation*}%
for all $t>0$.
\end{lemma}

\begin{proof}
From (\ref{Q}), (\ref{ddQ}) and (\ref{ddddQ}) it is obtained that%
\begin{eqnarray*}
q_{3}\left( t\right) &=&\tfrac{1}{2783\,424t^{5}\sinh ^{5}\frac{t}{2}}%
(-2087\,568t^{5}\cosh ^{4}\frac{t}{2}+1497\,636t^{5}\cosh ^{2}\frac{t}{2}%
\sinh ^{2}\frac{t}{2} \\
&&-165\,829t^{5}\sinh ^{4}\frac{t}{2}+399\,698t^{4}\sinh ^{5}\frac{t}{2} \\
&&+7502\,880t^{2}\sinh ^{5}\frac{t}{2}+66\,802\,176\sinh ^{5}\frac{t}{2}).
\end{eqnarray*}

We write $q_{3}\left( 2t\right) $ as 
\begin{eqnarray*}
&&\left( 2783\,424\sinh ^{5}t\right) q_{3}\left( 2t\right) \\
&=&-\left( 2087\,568\cosh ^{4}t+165\,829\sinh ^{4}t-1497\,636\cosh
^{2}t\sinh ^{2}t\right) \\
&&+199\,849\left( \sinh ^{4}t\right) \tfrac{\sinh t}{t}+937\,860\left( \sinh
^{2}t\right) \left( \tfrac{\sinh t}{t}\right) ^{3}+2087\,568\left( \tfrac{%
\sinh t}{t}\right) ^{5}.
\end{eqnarray*}%
Utilizing inequality (\ref{sinht/t>2}) and putting $\cosh t=x>1$ and then
factoring yield%
\begin{eqnarray*}
\left( 2783\,424\sinh ^{5}t\right) q_{3}\left( 2t\right) &>&-\left(
2087\,568x^{4}+165\,829\left( x^{2}-1\right) ^{2}-1497\,636x^{2}\left(
x^{2}-1\right) \right) \\
&&+199\,849\left( x^{2}-1\right) ^{2}\left( 15\frac{2x^{2}+10x+9}{%
2x^{2}+101x+212}\right) \\
&&+937\,860\left( x^{2}-1\right) \left( 15\frac{2x^{2}+10x+9}{2x^{2}+101x+212%
}\right) ^{3} \\
&&+2087\,568\left( 15\frac{2x^{2}+10x+9}{2x^{2}+101x+212}\right) ^{5} \\
&=&\frac{7\left( x-1\right) ^{5}}{\left( 2x^{2}+101x+212\right) ^{5}}%
q_{4}\left( x\right) >0,
\end{eqnarray*}%
where the last inequality holds due to%
\begin{eqnarray*}
q_{4}\left( x\right)
&=&10\,249\,024x^{9}+2015\,594\,800x^{8}+163\,876\,520\,192x^{7}+6681\,271%
\,280\,040x^{6} \\
&&+136\,012\,433\,414\,956x^{5}+1069\,481\,086\,377\,851x^{4}+4121\,483\,475%
\,973\,500x^{3} \\
&&+8450\,810\,874\,059\,188x^{2}+8899\,895\,239\,232\,240x+3802\,278\,457%
\,617\,584.
\end{eqnarray*}%
This completes the proof.
\end{proof}

\section{Main Results}

Now we state and prove the first result, which shows that the second
conjecture posed by Chen is valid.

\begin{theorem}
\label{MT-ha}Let the function $R$ be defined on $\left( 0,\infty \right) $
by (\ref{R(x)}). Then the function%
\begin{equation*}
h_{a}\left( x\right) =(x+a)^{2}R\left( x\right)
\end{equation*}%
is strictly completely monotonic on $(0,\infty )$ if $a\geq a_{0}=\sqrt{%
c_{0}^{2}+7/40}-c_{0}\approx 0.48476$, where $c_{0}\approx -0.061875$ is
defined by (\ref{c0}).
\end{theorem}

\begin{proof}
Using the relations (\ref{R0})--(\ref{R2}) we get that%
\begin{eqnarray*}
h_{a}\left( x\right) &=&(x+a)^{2}R\left( x\right) =x^{2}R\left( x\right)
+2axR\left( x\right) +a^{2}R\left( x\right) \\
&=&\frac{1}{24}+\int_{0}^{\infty }e^{-xt}Q^{\prime \prime }\left( t\right)
dt+2a\int_{0}^{\infty }e^{-xt}Q^{\prime }\left( t\right)
dt+a^{2}\int_{0}^{\infty }e^{-xt}Q\left( t\right) dt \\
&=&\frac{1}{24}+\int_{0}^{\infty }e^{-xt}\left( Q^{\prime \prime }\left(
t\right) +2aQ^{\prime }\left( t\right) +a^{2}Q\left( t\right) \right) dt \\
&&\overset{\vartriangle }{=}\frac{1}{24}+\int_{0}^{\infty }e^{-xt}Q\left(
t\right) \delta _{a}\left( t\right) dt.
\end{eqnarray*}%
Evidently, for $t>0$, $Q\left( t\right) >0$ and by Lemmas \ref{Lemma dQ/Q}
and \ref{Lemma ddQ/Q>}, 
\begin{eqnarray*}
\delta _{a}\left( t\right) &=&\frac{Q^{\prime \prime }\left( t\right) }{%
Q\left( t\right) }+2a\frac{Q^{\prime }\left( t\right) }{Q\left( t\right) }%
+a^{2}\geq a^{2}+2ac_{0}-\frac{7}{40} \\
&=&\left( a+c_{0}+\sqrt{c_{0}^{2}+\frac{7}{40}}\right) \left( a+c_{0}-\sqrt{%
c_{0}^{2}+\frac{7}{40}}\right) \geq 0
\end{eqnarray*}%
if $a\geq \sqrt{c_{0}^{2}+7/40}-c_{0}$, which proves the desired result.
\end{proof}

Taking $a=1/2$ and replacing $x$ by $\left( x+1/2\right) $ in the above
theorem, we have

\begin{corollary}
The function $\left( (x+1)/x\right) ^{2}H(x)$ is strictly completely
monotonic on $(-1/2,\infty )$.
\end{corollary}

\begin{remark}
Evidently, Theorem \ref{MT-ha} reveals that the second conjecture posed by
Chen in \cite[Theorem 2]{Chen-JMI-3(1)-2009} is true.
\end{remark}

Our second result states that

\begin{theorem}
\label{MT-Fa}Let the function $R$ be defined on $\left( 0,\infty \right) $
by (\ref{R(x)}). Then the function%
\begin{equation*}
x\mapsto F_{a}\left( x\right) =24\left( x^{2}+a\right) R\left( x\right) -1
\end{equation*}%
is strictly completely monotonic on $\left( 0,\infty \right) $ if and only
if $a\geq a_{1}=17/40$.
\end{theorem}

\begin{proof}
The necessity follows from%
\begin{equation*}
\lim_{x\rightarrow \infty }\frac{F_{a}\left( x\right) }{x^{-2}}%
=\lim_{x\rightarrow \infty }\frac{24\left( x^{2}+a\right) \left( \psi \left(
x+1/2\right) -\ln x\right) -1}{x^{-2}}=a-\frac{7}{40}\geq 0.
\end{equation*}%
Using the relations (\ref{R0}) and (\ref{R2}) we obtain that%
\begin{eqnarray*}
F_{a}\left( x\right) &=&24\left( x^{2}+a\right) R\left( x\right)
-1=24x^{2}R\left( x\right) +24aR\left( x\right) -1 \\
&=&24\left( \frac{1}{24}+\int_{0}^{\infty }e^{-xt}Q^{\prime \prime }\left(
t\right) dt\right) +24a\int_{0}^{\infty }e^{-xt}Q\left( t\right) dt-1 \\
&=&24\int_{0}^{\infty }e^{-xt}\left( Q^{\prime \prime }\left( t\right)
+aQ\left( t\right) \right) dt.
\end{eqnarray*}%
If $a\geq 7/40$, then by Lemma \ref{Lemma ddQ/Q>} it follows that 
\begin{equation*}
Q^{\prime \prime }\left( t\right) +aQ\left( t\right) \geq Q^{\prime \prime
}\left( t\right) +\frac{7}{40}Q\left( t\right) >0
\end{equation*}%
for $t>0$, which proves the sufficiency.
\end{proof}

Noting that%
\begin{equation*}
F_{7/40}\left( n+\frac{1}{2}\right) =24\left( \left( n+1/2\right)
^{2}+7/40\right) R_{n}-1
\end{equation*}%
and the facts that%
\begin{equation*}
F_{7/40}\left( 3/2\right) =\frac{286}{5}-\frac{291}{5}\ln \frac{3}{2}-\frac{%
291}{5}\gamma \approx 0.007979\text{ \ and \ }F_{7/40}\left( \infty \right)
=0,
\end{equation*}%
by Theorem \ref{MT-fa} we get immediately

\begin{corollary}
Let $R_{n}$ be defined by (\ref{R_n}). Then the double inequality%
\begin{equation}
\frac{1}{24\left( \left( n+1/2\right) ^{2}+7/40\right) }<R_{n}-\gamma <\frac{%
1+\lambda _{1}}{24\left( \left( n+1/2\right) ^{2}+7/40\right) }  \label{D1}
\end{equation}%
holds for $n\in \mathbb{N}$, where $\lambda _{1}=f_{7/40}\left( 3/2\right)
\approx 0.007979$ is the best constant.
\end{corollary}

Our third result is contained in the following theorem.

\begin{theorem}
\label{MT-fa}Let the function $R$ be defined on $\left( 0,\infty \right) $
by (\ref{R(x)}). Then the function%
\begin{equation*}
x\mapsto f_{a}\left( x\right) =-24\left( x^{4}+a\right) R\left( x\right)
+x^{2}-\frac{7}{40},
\end{equation*}%
is strictly completely monotonic on $\left( 0,\infty \right) $ is and only
if $a\leq a_{2}=-31/336$.
\end{theorem}

\begin{proof}
The necessity can be deduced by%
\begin{equation*}
\lim_{x\rightarrow \infty }\frac{f_{a}\left( x\right) }{x^{-2}}%
=\lim_{x\rightarrow \infty }\frac{-24\left( x^{4}+a\right) \left( \psi
\left( x+1/2\right) -\ln x\right) +x^{2}-\frac{7}{40}}{x^{-2}}=-a-\frac{31}{%
336}\geq 0.
\end{equation*}%
By the relations (\ref{R0}), (\ref{R2}) and (\ref{R4}), we obtain that%
\begin{eqnarray*}
f_{a}\left( x\right) &=&-24x^{4}R\left( x\right) -24aR\left( x\right) +x^{2}-%
\frac{7}{40} \\
&=&-24\left( \frac{1}{24}x^{2}-\frac{7}{960}+\int_{0}^{\infty
}e^{-xt}Q^{(4)}\left( t\right) dt\right) -24a\int_{0}^{\infty
}e^{-xt}Q\left( t\right) dt+x^{2}-\frac{7}{40} \\
&=&24\int_{0}^{\infty }e^{-xt}\left( -Q^{(4)}\left( t\right) -aQ\left(
t\right) \right) dt.
\end{eqnarray*}%
If $a\leq a_{2}=-31/336$, then by Lemma \ref{Lemma ddddQ/Q<} it follows that%
\begin{equation*}
-Q^{(4)}\left( t\right) -aQ\left( t\right) \geq -Q^{(4)}\left( t\right) +%
\frac{31}{336}Q\left( t\right) >0,
\end{equation*}%
which proves the sufficiency.
\end{proof}

Utilizing the decreasing property of $f_{a_{2}}$ and noting the facts that%
\begin{equation*}
f_{a_{2}}\left( \frac{3}{2}\right) =\frac{835}{7}\gamma +\frac{835}{7}\ln 
\frac{3}{2}-\frac{32\,819}{280}\approx 0.0090636\text{ and }f_{a_{2}}\left(
\infty \right) =0,
\end{equation*}%
we have

\begin{corollary}
Let $R_{n}$ be defined by (\ref{R_n}). Then the double inequality%
\begin{equation}
\frac{1}{24}\frac{\left( n+\frac{1}{2}\right) ^{2}-\frac{7}{40}-\lambda _{2}%
}{\left( n+\frac{1}{2}\right) ^{4}-\frac{31}{336}}<R_{n}-\gamma <\frac{1}{24}%
\frac{\left( n+\frac{1}{2}\right) ^{2}-\frac{7}{40}}{\left( n+\frac{1}{2}%
\right) ^{4}-\frac{31}{336}}  \label{D2}
\end{equation}%
holds for $n\in \mathbb{N}$, where $\lambda _{2}=f_{a_{2}}\left( 3/2\right)
\approx 0.0090636$ is the best constant.
\end{corollary}

\begin{remark}
The upper bound for $R_{n}-\gamma $ given in (\ref{D2}) is better than the
one given in (\ref{De2}), because%
\begin{eqnarray*}
&&\frac{1}{24}\frac{\left( n+\frac{1}{2}\right) ^{2}-\frac{7}{40}}{\left( n+%
\frac{1}{2}\right) ^{4}-\frac{31}{336}}-\left( \frac{1}{24n^{2}}+\frac{7}{960%
}\frac{1}{n^{4}}\right) \\
&=&-\frac{6720n^{5}+10\,752n^{4}+5712n^{3}+1564n^{2}+588n-35}{960n^{4}\left(
168n^{4}+336n^{3}+252n^{2}+84n-5\right) }<0.
\end{eqnarray*}
\end{remark}

The last result is the following theorem

\begin{theorem}
\label{MT-Ga}Let the function $R$ be defined on $\left( 0,\infty \right) $
by (\ref{R(x)}). Then the function%
\begin{equation}
x\mapsto G_{a}\left( x\right) =24\left( x^{4}+ax^{2}+\frac{7}{40}a-\frac{31}{%
336}\right) R\left( x\right) -\left( x^{2}-\frac{7}{40}+a\right)  \label{Ga}
\end{equation}%
is strictly completely monotonic on $\left( 0,\infty \right) $ if and only
if $a\geq a_{3}=11165/8284$.
\end{theorem}

\begin{proof}
The necessity can be deduced by%
\begin{eqnarray*}
\lim_{x\rightarrow \infty }\frac{G_{a}\left( x\right) }{x^{-4}}
&=&\lim_{x\rightarrow \infty }\tfrac{24\left( x^{4}+ax^{2}+\frac{7}{40}a-%
\frac{31}{336}\right) \left( \func{Psi}\left( x+1/2\right) -\ln x\right)
-\left( x^{2}-\frac{7}{40}+a\right) }{x^{-4}} \\
&=&\tfrac{2071}{33\,600}\left( a-\tfrac{11\,165}{8284}\right) \geq 0.
\end{eqnarray*}%
By the relations (\ref{R0}), (\ref{R2}) and (\ref{R4}) we obtain that%
\begin{eqnarray*}
G_{a}\left( x\right) &=&24x^{4}R\left( x\right) +24ax^{2}R\left( x\right)
+24\left( \frac{7}{40}a-\frac{31}{336}\right) R\left( x\right) -\left( x^{2}-%
\frac{7}{40}+a\right) \\
&=&24\left( \frac{1}{24}x^{2}-\frac{7}{960}+\int_{0}^{\infty
}e^{-xt}Q^{(4)}\left( t\right) dt\right) +24a\left( \frac{1}{24}%
+\int_{0}^{\infty }e^{-xt}Q^{\prime \prime }\left( t\right) dt\right) \\
&&+24\left( \frac{7}{40}a-\frac{31}{336}\right) \int_{0}^{\infty
}e^{-xt}Q\left( t\right) dt-\left( x^{2}-\frac{7}{40}+a\right) \\
&=&24\int_{0}^{\infty }e^{-xt}\left( Q^{(4)}\left( t\right) +aQ^{\prime
\prime }\left( t\right) +\left( \frac{7}{40}a-\frac{31}{336}\right) Q\left(
t\right) \right) dt \\
&&\overset{\vartriangle }{=}24\int_{0}^{\infty }e^{-xt}g_{a}\left( t\right)
dt.
\end{eqnarray*}

If $a\geq a_{3}=11165/8284$, then it follows from Lemmas \ref{Lemma ddQ/Q>}
and \ref{Lemma ddddQ-ddQ-Q} that%
\begin{eqnarray*}
g_{a}\left( t\right) &=&Q^{(4)}\left( t\right) -\frac{31}{336}Q\left(
t\right) +a\left( Q^{\prime \prime }\left( t\right) +\frac{7}{40}Q\left(
t\right) \right) \\
&\geq &Q^{(4)}\left( t\right) -\frac{31}{336}Q\left( t\right) +\tfrac{11\,165%
}{8284}\left( Q^{\prime \prime }\left( t\right) +\frac{7}{40}Q\left(
t\right) \right) \\
&=&Q^{(4)}\left( t\right) +\tfrac{11\,165}{8284}Q^{\prime \prime }\left(
t\right) +\frac{199\,849}{1391\,712}Q\left( t\right) >0,
\end{eqnarray*}%
which proves the sufficiency.
\end{proof}

Application of the decreasing property of $G_{a_{3}}$ and notice the facts
that%
\begin{equation*}
G_{a_{3}}\left( \frac{3}{2}\right) =\frac{112\,672\,809}{579\,880}-\frac{%
11\,465\,761}{57\,988}\ln \frac{3}{2}-\frac{11\,465\,761}{57\,988}\gamma
\approx 0.0016903\text{ and }G_{a_{3}}\left( \infty \right) =0,
\end{equation*}%
we have

\begin{corollary}
Let $R_{n}$ be defined by (\ref{R_n}). Then the double inequality%
\begin{equation}
\frac{1}{24}\tfrac{\left( n+\frac{1}{2}\right) ^{2}+\frac{97\,153}{82\,840}}{%
\left( n+\frac{1}{2}\right) ^{4}+\frac{11\,165}{8284}\left( n+\frac{1}{2}%
\right) ^{2}+\frac{199\,849}{1391\,712}}<R_{n}-\gamma <\frac{1}{24}\tfrac{%
\left( n+\frac{1}{2}\right) ^{2}+\frac{97\,153}{82\,840}+\lambda _{3}}{%
\left( n+\frac{1}{2}\right) ^{4}+\frac{11\,165}{8284}\left( n+\frac{1}{2}%
\right) ^{2}+\frac{199\,849}{1391\,712}}  \label{D3}
\end{equation}%
holds for $n\in \mathbb{N}$, where $\lambda _{3}=G_{a_{3}}\left( 3/2\right)
\approx 0.0016903$ is the best constant.
\end{corollary}

\section{Remarks}

\begin{remark}
\label{Remark Ga}The function $G_{a}$ defined by (\ref{Ga}) can be written as%
\begin{equation}
G_{a}\left( x\right) =a\times f_{7/40}\left( x\right) -F_{-31/336}\left(
x\right) =f_{7/40}\left( x\right) \times \left( a-\frac{F_{-31/336}\left(
x\right) }{f_{7/40}\left( x\right) }\right) .  \label{Ga-a}
\end{equation}

Theorem \ref{MT-Ga} tell us that%
\begin{equation}
\frac{F_{-31/336}\left( x\right) }{f_{7/40}\left( x\right) }\leq
\lim_{x\rightarrow \infty }\frac{F_{-31/336}\left( x\right) }{f_{7/40}\left(
x\right) }=\lim_{x\rightarrow \infty }\tfrac{-24\left( x^{4}-\frac{31}{336}%
\right) R\left( x\right) +x^{2}-\frac{7}{40}}{24\left( x^{2}+\frac{7}{40}%
\right) R\left( x\right) -1}\leq \frac{11\,165}{8284}.  \label{f/G<00}
\end{equation}%
On the other hand, we can prove that%
\begin{equation}
\frac{F_{-31/336}\left( x\right) }{f_{7/40}\left( x\right) }\geq
\lim_{x\rightarrow 0^{+}}\frac{F_{-31/336}\left( x\right) }{f_{7/40}\left(
x\right) }=\lim_{x\rightarrow 0^{+}}\tfrac{-24\left( x^{4}-\frac{31}{336}%
\right) R\left( x\right) +x^{2}-\frac{7}{40}}{24\left( x^{2}+\frac{7}{40}%
\right) R\left( x\right) -1}\geq \frac{155}{294}.  \label{f/G>0+}
\end{equation}%
It suffices to prove the function%
\begin{equation*}
x\mapsto V\left( x\right) =\psi \left( x+1/2\right) -\ln x-\frac{1}{24}\frac{%
x^{2}+\frac{2071}{5880}}{x^{2}\left( x^{2}+\frac{155}{294}\right) }
\end{equation*}%
is increasing on $\left( 0,\infty \right) $. Differentiation gives%
\begin{equation*}
V^{\prime }\left( x\right) =\psi ^{\prime }\left( x+1/2\right) -\frac{1}{x}+%
\frac{1728\,720x^{4}+1217\,748x^{2}+321\,005}{10x^{3}\left(
294x^{2}+155\right) ^{2}}.
\end{equation*}%
Utilizing $\psi ^{\prime }\left( x+1\right) -\psi ^{\prime }\left( x\right)
=-1/x^{2}$ yields%
\begin{eqnarray*}
V^{\prime }\left( x+1\right) -V^{\prime }\left( x\right) &=&-\tfrac{1}{%
\left( x+1/2\right) ^{2}}-\tfrac{1}{x+1}+\tfrac{1728\,720\left( x+1\right)
^{4}+1217\,748\left( x+1\right) ^{2}+321\,005}{10\left( x+1\right)
^{3}\left( 294\left( x+1\right) ^{2}+155\right) ^{2}} \\
&&+\tfrac{1}{x}-\tfrac{1728\,720x^{4}+1217\,748x^{2}+321\,005}{10x^{3}\left(
294x^{2}+155\right) ^{2}} \\
&=&-\tfrac{V_{1}\left( x\right) }{10x^{3}\left( 2x+1\right) ^{2}\left(
294x^{2}+155\right) ^{2}\left( x+1\right) ^{3}\left(
294x^{2}+588x+449\right) ^{2}},
\end{eqnarray*}%
where%
\begin{eqnarray*}
V_{1}\left( x\right)
&=&1718\,371\,882\,080x^{12}+10\,310\,231\,292\,480x^{11}+29\,399\,355\,669%
\,600x^{10} \\
&&+52\,486\,324\,833\,600x^{9}+66\,690\,983\,696\,400x^{8}+65\,258\,530\,001%
\,280x^{7} \\
&&+51\,909\,045\,513\,612x^{6}+34\,352\,301\,620\,196x^{5}+18\,881\,999\,450%
\,054x^{4} \\
&&+8378\,736\,976\,048x^{3}+2808\,871\,359\,013x^{2}+622\,502\,847\,155x+64%
\,714\,929\,005.
\end{eqnarray*}%
It is evident that $V_{1}\left( x\right) >0$ for $x>0$, and so $V^{\prime
}\left( x+1\right) -V^{\prime }\left( x\right) <0$ for $x>0$. This leads us
to%
\begin{equation*}
V^{\prime }\left( x\right) >V^{\prime }\left( x+1\right) >V^{\prime }\left(
x+2\right) >\cdot \cdot \cdot >\lim_{n\rightarrow \infty }V^{\prime }\left(
x+n\right) =0,
\end{equation*}%
which reveals that $V$ is increasing on $\left( 0,\infty \right) $.

Meanwhile, (\ref{Ga-a}) implies that an necessary condition such that the
function $-G_{a}$ is completely monotone on $\left( 0,\infty \right) $ is%
\begin{equation*}
a\leq \lim_{x\rightarrow 0+}\frac{F_{-31/336}\left( x\right) }{%
f_{7/40}\left( x\right) }=\lim_{x\rightarrow 0+}\tfrac{-24\left( x^{4}-\frac{%
31}{336}\right) R\left( x\right) +x^{2}-\frac{7}{40}}{24\left( x^{2}+\frac{7%
}{40}\right) R\left( x\right) -1}=\frac{155}{294}.
\end{equation*}%
This together with inequalities (\ref{f/G<00}) and (\ref{f/G>0+}) yield two
conjectures.
\end{remark}

\begin{conjecture}
Let the function $R$ be defined on $\left( 0,\infty \right) $ by (\ref{R(x)}%
). Then

(i) the function%
\begin{equation*}
x\mapsto \frac{-24\left( x^{4}-\frac{31}{336}\right) R\left( x\right) +x^{2}-%
\frac{7}{40}}{24\left( x^{2}+\frac{7}{40}\right) R\left( x\right) -1}
\end{equation*}%
is increasing on $\left( 0,\infty \right) $;

(ii) the function $-G_{a}$ is completely monotone on $\left( 0,\infty
\right) $ if and only if $a\leq 155/294$.
\end{conjecture}

\begin{remark}
\label{Remark D4}In addition, using the increasing property of the function $%
V$ proved in Remark \ref{Remark Ga}, and noting that%
\begin{equation*}
V\left( \frac{3}{2}\right) =\frac{866\,519}{881\,820}-\ln \frac{3}{2}-\gamma
\approx -0.000032387\text{ and }V\left( \infty \right) =0
\end{equation*}%
we have%
\begin{equation*}
\frac{1}{24}\tfrac{\left( n+\frac{1}{2}\right) ^{2}+\frac{2071}{5880}}{%
\left( n+\frac{1}{2}\right) ^{2}\left( \left( n+\frac{1}{2}\right) ^{2}+%
\frac{155}{294}\right) }+\lambda _{4}<R_{n}-\gamma <\frac{1}{24}\tfrac{%
\left( n+\frac{1}{2}\right) ^{2}+\frac{2071}{5880}}{\left( n+\frac{1}{2}%
\right) ^{2}\left( \left( n+\frac{1}{2}\right) ^{2}+\frac{155}{294}\right) },
\end{equation*}%
where $\lambda _{4}\approx -0.000032387$ is the best possible.

Clearly, it is an improvement of (\ref{D2}), since%
\begin{eqnarray*}
&&\frac{1}{24}\frac{\left( n+\frac{1}{2}\right) ^{2}+\frac{2071}{5880}}{%
\left( n+\frac{1}{2}\right) ^{2}\left( \left( n+\frac{1}{2}\right) ^{2}+%
\frac{155}{294}\right) }-\frac{1}{24}\frac{\left( n+\frac{1}{2}\right) ^{2}-%
\frac{7}{40}}{\left( n+\frac{1}{2}\right) ^{4}-\frac{31}{336}} \\
&=&-\frac{64\,201}{120\left( 2n+1\right) ^{2}\left( 588n^{2}+588n+457\right)
\left( 168n^{4}+336n^{3}+252n^{2}+84n-5\right) }<0.
\end{eqnarray*}
\end{remark}

\begin{remark}
Theorems \ref{MT-fa} and \ref{MT-Ga} and Remark \ref{Remark Ga} offer in
fact three new sequences convergent to $\gamma $ with increasingly higher
speed. Denote by%
\begin{eqnarray*}
w_{n} &=&\sum_{k=1}^{n}\frac{1}{k}-\ln \left( n+1/2\right) -\frac{1}{24}%
\frac{\left( n+\frac{1}{2}\right) ^{2}-\frac{7}{40}}{\left( n+\frac{1}{2}%
\right) ^{4}-\frac{31}{336}}, \\
y_{n} &=&\sum_{k=1}^{n}\frac{1}{k}-\ln \left( n+1/2\right) -\frac{1}{24}%
\tfrac{\left( n+\frac{1}{2}\right) ^{2}+\frac{97\,153}{82\,840}}{\left( n+%
\frac{1}{2}\right) ^{4}+\frac{11\,165}{8284}\left( n+\frac{1}{2}\right) ^{2}+%
\frac{199\,849}{1391\,712}}, \\
z_{n} &=&\sum_{k=1}^{n}\frac{1}{k}-\ln \left( n+1/2\right) -\frac{1}{24}%
\frac{\left( n+\frac{1}{2}\right) ^{2}+\frac{2071}{5880}}{\left( n+\frac{1}{2%
}\right) ^{2}\left( \left( n+\frac{1}{2}\right) ^{2}+\frac{155}{294}\right) }
\end{eqnarray*}%
Then by Theorems \ref{MT-fa} and \ref{MT-Ga} and Remark \ref{Remark D4} we
have%
\begin{equation*}
w_{n}<z_{n}<\gamma <y_{n}.
\end{equation*}%
And, a simple check yields%
\begin{eqnarray*}
\lim_{n\rightarrow \infty }n^{8}\left( w_{n}-\gamma \right) &=&-\frac{319}{%
92\,160}, \\
\lim_{n\rightarrow \infty }n^{10}\left( y_{n}-\gamma \right) &=&\frac{%
627\,404\,761}{246\,900\,842\,496}, \\
\lim_{n\rightarrow \infty }n^{8}\left( z_{n}-\gamma \right) &=&-\frac{%
199\,849}{94\,832\,640}
\end{eqnarray*}%
These show that the sequences $\left( w_{n}\right) $, $\left( y_{n}\right) $
and $\left( z_{n}\right) $ converge to $\gamma $ as $n^{-8}$, $n^{-10}$ and $%
n^{-8}$, respectively.
\end{remark}

Lastly, inspired by Theorem \ref{MT-Fa}--\ref{MT-Ga}, we post the following
open problem.

\begin{problem}
Determine the best constants $a_{k}$ and $b_{k}$ such that the function%
\begin{equation*}
x\mapsto \left( \sum_{k=0}^{n+1}a_{k}x^{2k}\right) R\left( x\right) -\left(
\sum_{k=0}^{n}b_{k}x^{2k}\right)
\end{equation*}%
is completely monotone on $\left( 0,\infty \right) $, and satisfies that%
\begin{equation*}
\lim_{x\rightarrow \infty }\frac{\left( \sum_{k=0}^{n+1}a_{k}x^{2k}\right)
R\left( x\right) -\left( \sum_{k=0}^{n}b_{k}x^{2k}\right) }{x^{-2n-4}}=c\neq
0,\pm \infty .
\end{equation*}
\end{problem}

\end{document}